\documentclass[a4paper,11pt,reqno]{amsart}
\usepackage[T1]{fontenc}
\usepackage[utf8]{inputenc}
\usepackage[margin=1.0in]{geometry}
\usepackage{lmodern}
\usepackage[english]{babel}
\usepackage{graphicx}
\usepackage{amsmath}
\usepackage{amsfonts}
\usepackage{amssymb} 
\usepackage{amsthm}
\usepackage{bbm}
\usepackage{bm} 
\usepackage{lipsum}
\usepackage{xcolor}
\usepackage{url}
\usepackage{hyperref}
\usepackage{stmaryrd}
\usepackage{comment}

\allowdisplaybreaks

\theoremstyle{plain} 
\newtheorem{thm}{Theorem}
\newtheorem{prop}[thm]{Proposition}

\newtheorem{ass}{Assumption}

\newcommand{\jo}[1]{{\textcolor{blue}{#1}}}

\newtheorem{innercustomgeneric}{\customgenericname}
\providecommand{\customgenericname}{}
\newcommand{\newcustomtheorem}[2]{%
  \newenvironment{#1}[1]
  {%
   \renewcommand\customgenericname{#2}%
   \renewcommand\theinnercustomgeneric{##1}%
   \innercustomgeneric
  }
  {\endinnercustomgeneric}
}

\newcustomtheorem{customthm}{Theorem}

\newtheorem{theorem}[thm]{Theorem}
\newtheorem{remark}[thm]{Remark}

\newtheorem{lemma}[thm]{Lemma}
\newtheorem{proposition}[thm]{Proposition}

\theoremstyle{definition}

\newtheorem{example}[thm]{Example}

\def \be {\begin{equation}}
\def \ee {\end{equation}}

\def \E {\mathbb{E}}
\def \P {\mathbb{P}}
\def \R {\mathbb{R}}

\def\red{\textcolor{red}}

\def \tra {\top}
\def \I {I}

\renewcommand{\phi}{\varphi}
\renewcommand{\epsilon}{\varepsilon}
\renewcommand{\tilde}{\widetilde}
\renewcommand{\hat}{\widehat}
\renewcommand{\bar}{\overline}

\DeclareMathOperator*{\argmin}{arg\,min}

\begin{document}

\title[Convergence for LQ potential MFG]{Convergence for linear quadratic potential mean field games}

\author{Alekos Cecchin}
\author{Jodi Dianetti}

\address[A. Cecchin]
{University of Padova, Department of Mathematics ``Tullio Levi-Civita'', 
 \newline \indent Via Trieste 63, 35121 Padova, Italy}
\email{alekos.cecchin@unipd.it}

\address[J. Dianetti]
{University of Rome Tor Vergata, Department of Economics and Finance,
\newline \indent Via Columbia 2, 00133 Roma, Italy
}
\email{jodi.dianetti@uniroma2.it}

\thanks{
A.C. acknowledges support from the project MeCoGa ``Mean field control and games'' of the University of Padova through the program STARS@UNIPD,
the PRIN 2022 Project 2022BEMMLZ ``Stochastic control and games and the role of information'', 
 the INdAM-GNAMPA Project 2025 ``Stochastic control and MFG under asymmetric information: methods and applications'' and the PRIN 2022 PNRR Project P20224TM7Z  ``Probabilistic methods for energy transition''.}

\date{\today}


\begin{abstract}
This paper studies the limits of empirical means of open-loop Nash equilibria of linear-quadratic stochastic differential games as the number of players goes to infinity, when the corresponding mean field game is of potential type and may have multiple equilibria. 
Via weak compactness arguments, the limit points are characterized as optimal trajectories of the related deterministic control problem, thus ruling out some of the mean field  equilibria.
Our result is obtained by first connecting the finite player game to a suitable control problem, whose optimal trajectories are the empirical means of Nash equilibria of the game, and in which the number of players $N$ becomes a parameter.  
True convergence to the unique minimizer of the limit control problem then holds for almost every initial mean.  
In cases of multiple optimizers, we focus on  examples to show that some   symmetry of the data ensures that the sequence admits a random limit which is  distributes uniformly among the minimizers of the potential. 
Multidimensional examples of the convergence result appear here for the first time, which show the flexibility of our method. 
We also establish a similar convergence results for the corresponding linear-quadratic potential mean field games with common noise, as the noise vanishes. 
\end{abstract}

\maketitle
 
\tableofcontents

\section{Introduction}
Mean field games (MFGs) were introduced independently in \cite{HuangMalhameCaines06} and \cite{LasryLions07} as limit models for symmetric stochastic differential games with mean field interaction, as the number of players tends to infinity. 
This approach provides a way of analyzing games with a large number of players, which are otherwise intractable due to the curse of dimensionality that arises when numerically solving the equations associated with Nash equilibria. This line of research has culminated in several textbooks (see e.g. \cite{cardaliaguet.delarue.lasry.lions.2019, CarmonaDelarue18, achdou2021mean}), which survey and synthesize the many advances that MFGs have brought to probability theory, analysis, and applications in the social sciences.

The connection between the MFG and the related 
$N$-player game can be understood in two directions. 
On the one side, approximation results show that each mean field equilibrium provides an approximate Nash equilibrium, with an error that vanishes as the number of players $N$ tends to infinity (see e.g. \cite{carmona.delarue.2013, HuangMalhameCaines06}).  
On the other side, convergence results provide sufficient conditions under which any sequence of Nash equilibria for the $N$-player game converges (up to subsequences) to a MFG equilibrium. 
However, MFGs may admit equilibria that are not limit points of any sequence of Nash equilibria.  
In other words, while every mean field equilibrium provides an approximate Nash equilibrium, some MFG equilibria are not informative about the true Nash equilibria, and the equilibrium behavior of the $N$-player system may look very different from that of certain equilibria of the limit model; see \cite{cardaliaguet.delarue.2022selected, graber2025remarks} for a recent survey.

This paper addresses the problem of determining which mean field equilibria arise as limit points of Nash equilibria in potential linear--quadratic models.
Through the lens of the potential, our main results characterize the limits of Nash equilibria as minimizers of a limiting control problem.
These results are then used to construct the limiting distribution in a series of examples, where certain symmetries in the model data are present and multiple MFG equilibria arise.
 
\subsection{Background and motivation}
Solving a mean field game typically amounts to solving a fixed point problem, which can be formulated either via a system of partial differential equations (see e.g. \cite{LasryLions07, achdou2021mean})),  via the best response map  (see e.g. \cite{dianetti.ferrari.fischer.nendel.2019, Lacker15}), or via the probabilistic system of forward-backward stochastic differential equations (FBSDEs) (see \cite{ahuja.ren.yang.2019, carmona.delarue.2013, dianetti2022strong}). 
Relying on the continuity of the problem and on structural conditions on the game, existence of solutions typically holds and  can be addresses via topological fixed point theorems (see \cite{CarmonaDelarueLacker16, Lacker15}), or via lattice theoretical fixed point theorems (see \cite{carmona.delarue.lacker.2017.timing, dianetti2022strong, dianetti.ferrari.fischer.nendel.2019, dianetti.ferrari.fischer.nendel.2022unifying}).
%
The search for equilibria relies on different techniques in the case of MFGs of \emph{potential} type, that is, when the data of the model are given by derivatives of some functions with respect to the measure variable.
In this case, the game can be linked to a mean field control problem: 
any solution of the control problem is a MFG equilibrium, since the
system of forward-backward equations of the MFG represents the necessary conditions for optimality given by the maximum principle; see \cite{briani2018stable, carmona.delarue.2015forward,   cecchin2025weak, LasryLions07}.

The uniqueness of the equilibrium holds for small time horizons (see \cite{BardiFischer18, HuangMalhameCaines06}); however,  uniqueness is not guaranteed in general, and it depends on structural properties of the game.
These properties are often referred to as monotonicity conditions, and several versions have been proposed (see e.g. \cite{achdou2021mean, ahuja.ren.yang.2019, gangbo.meszaro.2022.global, gangbo.meszaros.mou.zhang.2022, mou.zhang.2022master.antimon}), the most common being the Lasry–Lions monotonicity condition, already employed in \cite{LasryLions07}; 
in a nutshell, such conditions mean that players prefer to spread instead of aggregate. 
A particular case of study is that of linear-quadratic MFGs; see \cite{bensoussan2016linear, huang2021linear}. 
Uniqueness typically holds for purely linear-quadratic MFGs, as the system of equations becomes finite-dimensional, and a form of monotonicity is verified. 
%
In contrast, cases in which multiple equilibria arise are common in several applications, when any monotonicity condition is violated.
Several examples of non-uniqueness have been proposed; see \cite{BardiFischer18, bayraktar.zhang.2020, Cecchin&DaiPra&Fischer&Pelino19, DelarueFT19}.  
All these examples satisfy the \emph{submodularity condition}, which is a structural condition in which players have an incentive to aggregate, and  which can be seen as a sort of antithetic version of the Lasry–Lions monotonicity.
Under submodularity, a lattice structure on the set of (possibly multiple) equilibria can be obtained; see \cite{carmona.delarue.lacker.2017.timing, dianetti2022strong, dianetti.ferrari.fischer.nendel.2019, dianetti.ferrari.fischer.nendel.2022unifying,mou.zhang.2024minimal}.
Moreover, when the strength of the interaction among players is sufficiently high, uniqueness is recovered; see \cite{ mou.zhang.2022master.antimon, dianetti.ferrari.federico.floccari.2024multiple}.

The multiplicity of equilibria is a crucial issue for the convergence problem.
The search for Nash equilibria of a $N$-player game, and also the question of convergence as $N$ grows, must take into account the information available to players, and is thus different if considering open-loop or closed-loop equilibria. As a general result, even in presence of multiple MFG equilibria, any sequence of (approximate) Nash equilibria, either open-loop or closed-loop, can be shown to be tight and have limit points supported on the set of weak MFG equilibria; 
see \cite{djete2023large,fischer.2017connection, lacker.2016general, lacker.leflem.2023closedloop.convergence}. If the MFG equilibrium is unique, other techniques can be employed to leverage the convergence to a quantitative result. Open-loop Nash equilibria can be characterized by a system of FBSDE, whose convergence is guaranteed under some form of monotonicity or smallness condition; see \cite{carmona.delarue.2015forward, lauriere2022convergence}.
On the other hand, closed-loop equilibria can be characterized by a system of PDEs (called the Nash system), whose convergence is established if it is possible to construct a solution to the master equation which is at least continuous: this is an equation which serves as the decoupling field of the forward-backward MFG system of PDEs or SDEs, and thus a regular solution can be defined when such system is uniquely solvable; see \cite{cardaliaguet.delarue.lasry.lions.2019, mou2024wellposedness}. 
Similar techniques have been employed to study convergence of Nash equilibria of linear-quadratic $N$-player games, in cases uniqueness holds in the limit; see \cite{bardi2014linear, cirant2025some, li.mou.wu.zhou.2023linearquadratic}. 

The convergence problem in case of multiple equilibria presents many open problems; we refer again to \cite{cardaliaguet.delarue.2022selected, graber2025remarks} for a survey.  
While any sequence of Nash equilibria converges, up to subsequences, to a weak mean field equilibrium, in general not all mean field equilibria arise as limit points of such sequences.
The question of which mean field equilibria is informative about the $N$-player game thus remains a central question in the theory and largely open. 
True convergence of the sequence of Nash equilibria to a unique mean field equilibrium or to a randomization of some of them is shown just in two particular examples:  a selection principle is established for open-loop equilibria in a linear-quadratic one-dimensional model in \cite{DelarueFT19}, and for closed-loop equilibria in a two-state model in \cite{Cecchin&DaiPra&Fischer&Pelino19}; 
these results hinge on a fine analysis and techniques which are specific of the considered models and difficult to generalize in other contexts.

A natural conjecture can be formulated for the convergence problem for potential MFGs. 
In that case, the optimizers of the additional control problem are also solutions to the MFG, but there are MFG equilibria which are not optimizers, but just stationary points.  
The conjecture is that limit points of Nash equilibria are supported on optimizers of the mean field control problem, thus ruling out equilibria which are not optimizers; this is shown just in the two particular one-dimensional examples \cite{Cecchin&DaiPra&Fischer&Pelino19,DelarueFT19}. Notably, the optimizer is typically unique in an open and dense set of probability measures \cite{cardaliaguet2023regularity}. 
We mention also the convergence problem for mean field control problems, for which the prelimit model is a cooperative $N$-player game; we refer to \cite{cardaliaguet2023sharp, cardaliaguet2023regularity} for recent results. 
However, the cooperative and non-cooperative $N$-player games are not related and in general not close; indeed, the $N$-player non-cooperative game (which is the prelimit of MFGs) is not, by its nature, of potential type. 
See the related notion of $\alpha$-Potential game recently introduced in \cite{ guo.zhang.2025potential, guo.li.zhang.2025distributed}. 
We also refer to  the recent \cite{campi.cannerozzy.cartellier.2025coarse, cannerozzi.ferrari.2026cooperation}, which compare cooperation and non-cooperation in the MFG context.

A related convergence problem is that of vanishing common noise. 
We refer to    
\cite{CarmonaDelarueLacker16, dianetti2022strong, gangbo.meszaros.mou.zhang.2022} for general existence results for MFGs with common noise, and to \cite{delarue2019restoring} for a case in which a form of noise restores uniqueness in presence of multiple MFG equilibria. 
For linear-quadratic MFGs, the finite-dimensional system becomes parabolic in presence of the usual common noise, and thus is shown to be uniquely solvable in \cite{tchuendom}. When the common noise restores uniqueness, one wonders about selection principles for MFG solutions as the noise vanishes. Results in this direction can be found in \cite{DelarueFT19} for a one-dimensional linear-quadratic example, and in \cite{cecchin.delarue.2021} for a model of potential finite state MFGs. 

These convergence problems can also be framed into the more general context of zero-noise limit of a system. 
Indeed, it is a well known fact in probability theory that certain deterministic systems become well posed when a small noise is plugged, even if the original deterministic system admits multiple solutions. 
However, finding the limit of the solution as the noise vanishes represents a challenging question, for which few results are known: we refer the interested reader to \cite{bafico.baldi.1982small, delarue.flandoli.2014transition, trevisan.2013zero} for one-dimensional models, and to \cite{delarue2019zero} for a potential multi-dimensional model. 


\subsection{Our results and methodology} 
We focus on potential MFGs in which multiple MFG equilibria arise. The cost we consider is linear-quadratic in the control variable and is a general potential function (not linear quadratic) of the empirical mean; thus our model is a perturbation of a linear-quadratic MFG. Our method consists in defining a control problem related to the open-loop Nash equilibrium of the $N$-player game, which has the Nash FBSDE system as the FBSDE derived via the Pontryagin stochastic maximum principle.  
More precisely, by averaging on the player's indexes the components of the solution to the Nash FBSDE, the linear quadratic structure allows to obtain another FBSDE which is written in terms of the empirical mean (i.e., the arithmetic mean of the states of the players at equilibrium).  
Such a system of FBSDEs is indeed the Pontryagin system related to a stochastic control problem in which the state process is the empirical mean, and with costs written in terms of the integral of the original costs and of some reminder costs. 
In a nutshell, we obtain a control problem whose minimizer is the empirical mean of the Nash equilibrium. 
We mention that, differently from other works, the cost of a player in the $N$-player depends on the other players through the empirical mean of the whole system, not through the empirical mean of the other players excluding the private state. 


Taking limits as $N$ goes to infinity, weak compactness arguments allow to show that every limit point of the sequence of empirical means related to Nash equilibria is supported on minimizers of a deterministic control problem; indeed, in linear-quadratic cases the reminder costs are shown to go to zero. Stationary points of this limit control problem are shown to be expectations of MFG solutions, thanks to Pontryagin's maximum principle. Hence, we show a selection principle as limits of empirical means are supported on optimizers, thus ruling out (expectation of) MFG equilibria which are not minimizers.  In particular, optimizers of deterministic control problem are known to be unique for almost every initial point, thus yielding a true convergence result to a unique deterministic limit. This also provides almost sure convergence of the corresponding decoupling fields, which are smooth solutions to a system of PDEs. 
We prove  similar results also when analyzing the vanishing common noise; in that case, there is no additional reminder in the cost.

When optimizers are not unique, we wonder whether the sequence of empirical means admits a true random limit which is supported in (some) optimizers. 
We then turn our focus on examples. 
When the potential is convex, the known convergence result under uniqueness of the equilibrium is recovered. 
When the potential is not convex, we study some examples with submodular structure in which multiple equilibria arise and use the previous convergence result to characterize the limiting distribution. 
In our examples, the symmetry of the data ensures that the limiting measure distributes uniformly among the minimizers of the potential. 
The exact characterization allows to show convergence without going through a subsequence and we partially recover the result of \cite{DelarueFT19}.   
Multidimensional examples of the convergence result appear here for the first time, which show the flexibility of our method when compared with the techniques in \cite{DelarueFT19}.

\subsection{Outline} 
The paper is organized as follows. In Section \ref{sec:LQMFG} we first describe the $N$-player games and their open-loop equilibria in \S \ref{sec:Nplayer}, then their related control problem for the empirical mean in \S \ref{sec:N-control}, and thus the potential MFG in \S \ref{sec:potentialMFG}.  Section \ref{sec:N-convergence} presents the main result on convergence of the $N$-player empirical means: see Theorem \ref{thm:convergence}. Section \ref{sec:MFG-common-noise} examines the potential MFG with common noise and the related convergence problem in Theorem \ref{thm:eps_convergence}.  Section \ref{sec:examples} presents the the convergence results for some models whereas optimizers of the deterministic control problem are not unique: a one-dimensional symmetric example is analyzed in Theorem \ref{prop case two minima general}, a submodular example is treated in Proposition \ref{prop symmetric case two minima}, and a multi-dimensional symmetric model is studied in Theorem \ref{prop multidimensional}.

\section{Linear quadratic mean field games and related control problems} 
\label{sec:LQMFG}

\subsection{The $N$-player game}
\label{sec:Nplayer}
On a complete probability space $(\Omega, \mathcal F, \mathbb P)$, consider $N$ independent $d$-dimensional  Brownian motions $W^1,...,W^N$ and $N$ independent $d$-dimensional square integrable random variables $\xi^1,...,\xi^N$, with $\xi^i$ being independent from $W^j$ for each $i,j=1,...,N$. 
Set $\bm{W}=(W^{1}, \ldots, W^{N})$, $\bm{\xi}=(\xi^{1}, \ldots, \xi^{N})$ and denote by  $\mathbb{F}^{\bm{W}, \bm{\xi}}$ the right continuous extension of the filtration generated by $\bm{W}$ and $\bm{\xi}$, augmented by the $\mathbb P$-null sets. 
Denote by $\mathcal A_N$ the space of (open-loop) controls; that is,  the space of $\mathbb R^d$-valued $\mathbb{F}^{\bm{W}, \bm{\xi}}$-progressively measurable square integrable processes.  

Consider $N$ players, indexed by $i=1,...,N$.
For $i=1,...,N$, when player $i$ chooses a (open-loop) control $\alpha^i \in \mathcal A_N$, its own state $X^i$ is determined as  the solution to the SDE
\be
\label{eq:SDE:N}
d X_{t}^{i}= (b X^i_t + \alpha^i_t ) d t+ \sigma d W_{t}^{i} , \quad X_{0}^{i}=\xi^{i},
\ee
for a matrix $b \in \R^{d\times d}$ and a constant $\sigma \in \mathbb R$.
The aim of player $i$ is to choose $\alpha^i$ in order to minimize the cost  
\be 
\label{eq:cost:N}
J^{i}(\alpha^i,\bm{\alpha}^{-i})=\mathbb{E}\bigg[\int_{0}^{T} \frac{1}{2} \left(|\alpha_{t}^{i} |^{2}+ \big(X_{t}^{i} + \nabla f(\begin{matrix} \frac{1}{N} \sum_{j=1}^{N} X^j_t \end{matrix}) \big)^2 \right) d t
+ \frac12 \big( X_{T}^{i} + \nabla g 
(\begin{matrix} \frac{1}{N} \sum_{j=1}^{N} X^j_T \end{matrix}) \big)^2\bigg],
\ee
where $\bm{\alpha}^{-i} = (\alpha^j)_{j\ne i}$ denotes the vector of strategies chosen the all the other players and $f,g: \mathbb{R}^{d} \rightarrow \mathbb{R}$ are $C^{2}$ functions, with (column) gradients $\nabla f, \nabla g$.

\begin{remark}
    We point out that our $N$-player game slightly differs from those considered in many related papers (see e.g. \cite{DelarueFT19, li.mou.wu.zhou.2023linearquadratic}), where player $i$'s optimization problem depends instead  on $\nabla f \big( \frac{1}{N-1} \sum _{j\ne i} X_t^j)$ and $\nabla g \big( \frac{1}{N-1} \sum _{j\ne i} X_T^j)$, thus being convex in player $i$'s own state $X^i$.
    Indeed, in our setup player $i$ optimization problem is not convex, and the related FBSDE is carries an extra non-linearity in the empirical mean $\frac{1}{N} \sum_{j=1}^{N} X^j$.

     We also notice that the cost is a nonlinear function of the empirical mean, thus it is a perturbation of a linear-quadratic model. 
    
    We finally mention that the case in which the cost is just of the form $x^i \nabla f(\tfrac1n \sum_{j=1}^N x^j)$ is not a particular case of what we consider, but the same results we provide in the paper could be derived in the same way for such formulation.
\end{remark}

An open-loop Nash equilibrium is a vector of strategies $\bm{\alpha} = (\alpha^1,...,\alpha^N)$ such that
$$
J^{i}(\alpha^i,\bm{\alpha}^{-i}) \leq J^{i}(\beta^i,\bm{\alpha}^{-i}), \quad \forall \ \beta^i \in \mathcal A , 
\quad i = 1,...,N.
$$
Thanks to the stochastic maximum principle (see e.g. \cite{CarmonaDelarue18}), one can determine the Nash equilibria of the game by solving the following system of FBSDEs:
\begin{equation}
\label{eq:FBSDE:N}
\left\{
\begin{array}{l}
d X_{t}^{i}=( b X^i_t-Y_{t}^{i}) d t+\sigma d W_{t}^{i}, \qquad X_{0}^{i}=\xi^{i},  \\
 d Y_{t}^{i}=- \Big[ b^\tra Y^i_t + 
 \Big( \I_d + \frac{1}{N} \nabla^{2} f\big(\frac{1}{N} \sum_{j} X_{t}^{j}\big) \Big) \Big( X^i_t +  \nabla f \big(\frac{1}{N} \sum_{j} X_{t}^{j}\big) \Big)    \Big] dt+\sum_{j=1}^{N} Z_{t}^{i j} d W_{t}^{j}, \\
Y_{T}^{i}=\Big( \I_d + \frac{1}{N} \nabla^{2} g \big(\frac{1}{N} \sum_{j} X_{T}^{j}\big) \Big) \Big( X^i_T +  \nabla g \big(\frac{1}{N} \sum_{j} X_{T}^{j}\big) \Big),   
\end{array}\right.
\end{equation}
where $\nabla^2 f, \nabla^2 g $ denote the Hessian matrices and we consider the derivative of the functions
$$
\begin{aligned}
 F^{i}(\bm{x})&= \frac12 \Big( x_{i} + \nabla f\big(\tfrac{1}{N} \sum_{j} x_{j}\big) \Big)^2, \qquad  \bm{x}=\left(x_{1},\dots, x_{N}\right), \quad x_{i} \in \mathbb{R}^{d}, \\
\frac{\partial F^i}{\partial x_{i}}(\bm{x})&= x_i + 
\nabla f \big(\tfrac{1}{N} \sum_{j} x_{j}\big)+\frac{1}{N} \nabla^{2} f \big(\tfrac{1}{N} \sum_{j} x_{j}\big) x_i  
+ \frac1N  \nabla^2 f \big(\tfrac{1}{N} \sum_{j} x_{j}\big) \nabla f \big(\tfrac{1}{N} \sum_{j} x_{j}\big),
\end{aligned}
$$
and 
$$
\begin{aligned}
 G^{i}(\bm{x})&= \frac12 \Big( x_{i} + \nabla g\big(\tfrac{1}{N} \sum_{j} x_{j}\big) \Big)^2, \qquad  \bm{x}=\left(x_{1},\dots, x_{N}\right), \quad x_{i} \in \mathbb{R}^{d}, \\
\frac{\partial G^i}{\partial x_{i}}(\bm{x})&= x_i + 
\nabla g \big(\tfrac{1}{N} \sum_{j} x_{j}\big)+\frac{1}{N} \nabla^{2} g \big(\tfrac{1}{N} \sum_{j} x_{j}\big) x_i  
+ \frac1N  \nabla^2 g \big(\tfrac{1}{N} \sum_{j} x_{j}\big) \nabla g \big(\tfrac{1}{N} \sum_{j} x_{j}\big).
\end{aligned}
$$
A solution to this system is denoted by a tuple $(\bm{X}, \bm{Y}, \bm{Z}) = (X^1,...,X^N,Y^1,...,Y^N,Z^1,...,Z^N)$, where for any $i$, $X^i,Y^i$ are $\R^d$-valued, $Z^i = (Z^{ij})_{j=1,...,N}$ with $Z^{ij}$ being valued in the space of $\R^{d \times d}$-matrices.

Given a solution $(\bm{X}, \bm{Y}, \bm{Z})$ to the Nash FBSDE \eqref{eq:FBSDE:N}, define the triple $(m^{N}, \eta^{N}, \zeta^{N}) = (m^{N}, \eta^{N}, \zeta^{1,N},...,\zeta^{N,N}) $ of empirical mean as
$$
m^N_t = \frac{1}{N} \sum_{i=1}^N X_{t}^{i}, \quad  \eta_{t}^{N}=\frac{1}{N} \sum_{i=1}^N Y_{t}^{i}
\quad \text{and}\quad
\zeta_{t}^{j, N}=\frac{1}{N} \sum_{i=1}^N Z_{t}^{i j}, \ j=1,...N.
$$
Then, $(m^{N}, \eta^{N}, \zeta^{N})$ solves the FBSDE
\be 
\label{eq:FBSDE:mean}
\left\{\begin{array}{l} 
d m_{t}^{N}= (b m^N_t-\eta_{t}^{N} ) dt+\frac{\sigma}{N} \sum_{j} d W_{t}^{j},   \\
d \eta_{t}^{N}= - \Big[ b^\tra \eta^N_t + \Big( \I_d + \frac{1}{N} \nabla^{2} f(m^N_t) \Big) \Big( m^N_t +  \nabla f (m^N_t) \Big) \Big] dt +\sum_{j} \zeta_{t}^{j,N} d W_{t}^{j}, \\
m^N_{0}=\frac{1}{N} \sum_{j} \xi^{i}, \qquad \eta_{T}^{N}=\Big( \I_d + \frac{1}{N} \nabla^{2} g(m^N_T) \Big) \Big( m^N_T +  \nabla g (m^N_T) \Big).
\end{array}\right.
\ee 

Our analysis hinges on  an existence and uniqueness result for the FBSDE \eqref{eq:FBSDE:mean}.
To this aim, we introduce the following requirements, which will be standing throughout the paper.
\begin{ass} 
    The functions $f$ and $g$ are in $C^2(\R^d)$, and the functions  $$ y \mapsto \nabla f (y), \nabla^2 f (y), \nabla^2 f(y) y,\nabla g (y), \nabla^2 g (y), \nabla^2 g(y) y$$ are bounded and Lipschitz continuous. 
\end{ass}
    
The proof of the following lemma, which is given in Appendix \ref{appendix lemma existence uniqueness FBSDE.mean}, 
exploits a transformation of the system using the associated Riccati equation (similarly, e.g., to the arguments already employed in \cite{DelarueFT19, tchuendom}); we remark that this is not a standard linear system. The assumption that $\nabla^2 f(y) y$, $\nabla^2 g(y) y$ are bounded is required only for this well-posedness result; notice that in the particular case of dimension one, such  conditions is not necessary. 

\begin{thm}\label{lemma existence uniqueness FBSDE.mean}
    For $\sigma >0$, the FBSDE \eqref{eq:FBSDE:mean} admits a unique solution.
\end{thm}
  
Let us recall that we study the strong formulation of the control problems  and thus we always consider strong solutions to FBSDEs. As a consequence, we obtain an existence and characterization result of the unique Nash equilibrium, for which we provide a proof in Appendix \ref{appendix thm:EU NE}.

\begin{thm}\label{thm:EU NE}
For $\sigma >0$ and $N\geq||\nabla^2 f||_\infty\vee||\nabla^2 g||_\infty$, the FBSDE \eqref{eq:FBSDE:N} admits a unique solution  $(\bm{X}, \bm{Y}, \bm{Z})$. Moreover, there exists a unique Nash equilibrium $\bm{\alpha} = (\alpha^1,...,\alpha^N)$; it is given by  $\alpha^i = -Y^i$ for any $i=1,...,N$.
\end{thm}


\subsection{A related control problem}
\label{sec:N-control} 
Despite the $N$-player game introduced above is not of potential type, our aim is to view the empirical  mean $m^N$ associated to the Nash equilibrium as the minimizer of an auxiliary optimal control problem to be determined.
Indeed, the fact that $m^N$ solves the FBSDE \eqref{eq:FBSDE:mean} suggests to search for a stochastic control problem for which \eqref{eq:FBSDE:mean} is the related FBSDE.
To this end, set 
$$
\bar{\xi}^N  = \frac{1}{N} \sum_{i=1}^N \xi^{i}
\quad \text{and}\quad
\bar{W}^N=\frac{1}{\sqrt{N}} \sum_{i=1}^N W^{i},
$$
and consider the optimal control problem which consists in minimizing, over $\beta \in \mathcal A _N$, the cost functional
\begin{equation}
\label{eq:OC:N}
\tag{OC$_{N}$}
\begin{split}
& J _c^N(\beta)
=\mathbb{E} \bigg[ \int_{0}^{T}  \Big( \frac{1}{2}  \left|\beta_{t}\right|^{2} + \frac12 |m_t|^2 + f\left(m_{t} \right)  +\frac{1}{N}  R_f (m_t) \Big) dt  + \frac12 |m_T|^2 + g\left(m_{T}\right) +\frac{1}{N}R_g(m_{T}) \bigg], \\
&\text{subject to} \quad d m_{t}= (bm_t +\beta_{t}) dt +\frac{\sigma}{\sqrt{N}} d \bar{W}^N_{t},  \quad m_0 = \bar \xi^N,
\end{split}
\end{equation}
where the reminders $R_f,R_g$ are defined as
\be
\begin{split}
    R_f (m) & :=  \tfrac12 |\nabla f (m)|^2 + m \cdot \nabla f\left(m \right)-f\left(m\right), \\
    R_g (m) & := \tfrac12 |\nabla g (m)|^2 + m \cdot \nabla g\left(m \right)-g\left(m\right).
\end{split}
\ee
A control $\beta \in \mathcal A _N$ is said to be optimal if $J_c^N(\beta) \leq J_c^N(\beta') $ for any $\beta' \in \mathcal A_N$.
If $\beta$ is optimal, we refer to the associated state process $m$ as the optimal trajectory (related to $\beta$), and to $(\beta, m)$ as an optimal pair.

The following result connects the control problem \eqref{eq:cost:N} to the $N$-player game. 
\begin{thm}
\label{thm: Nash corresponds to minima of OC}
For $\sigma >0$, the following statements hold true:
\begin{enumerate}
    \item There exists a unique optimal control of the control problem \eqref{eq:OC:N}; 
    \item The optimal pair of  \eqref{eq:OC:N} is given by  the empirical mean $(-\eta^N,m^N)$ associated to the unique Nash equilibrium.
\end{enumerate}
\end{thm}

\begin{proof} 
The costs of the control problem \eqref{eq:OC:N} are
$$
\begin{aligned}
    F_N(m) &= \frac12 |m|^2 + f(m)+\frac{1}{N} R_f (m ),\\
    G_N(m) &= \frac12 |m|^2 + g(m)+\frac{1}{N} R_g (m ),
\end{aligned}
$$
and computing the gradients in $m$ we find
$$
\begin{aligned}
\nabla F_N(m)&=
\Big( \I_d + \frac{1}{N} \nabla^{2} f(m) \Big) \Big( m +  \nabla f (m) \Big), \\
\nabla G_N(m)&=
\Big( \I_d + \frac{1}{N} \nabla^{2} g(m) \Big) \Big( m +  \nabla g (m) \Big).
\end{aligned}
$$
By the stochastic maximum principle, any optimal control of the control problem \eqref{eq:OC:N}, writes as $-\eta$ for some solution $(m,\eta, \zeta )$ to  the FBSDE
\be\label{eq:FBSDE.mean.niceFG}
 \left\{\begin{array}{l}
 d m_{t} = (b m_t -\eta_{t}) d t+ \sigma d \bar W ^N_t
 \quad m_{0}= \bar \xi ^N, \\ 
d \eta_{t} =-[ b^\tra \eta_t + \nabla F_N \left( m_{t} \right) ] dt +\sum_{j} \zeta_{t}^{j} d W_{t}^{j}, \quad \eta_{T}=\nabla G_N \left(m_{T}\right).
\end{array}\right.
\ee
Such a FBSDE coincides with the FBSDE \eqref{eq:FBSDE:mean}, hence it admits a unique solution $(m,\eta, \zeta )$  by Theorem \ref{lemma existence uniqueness FBSDE.mean}.
Thus the control problem \eqref{eq:OC:N} admits a unique optimal control $-\eta$, proving $(1)$.

To show $(2)$, it is enough to notice that the  FBSDE
\eqref{eq:FBSDE.mean.niceFG} coincides with the FBSDE \eqref{eq:FBSDE:mean}, for which $(m^N,\eta^N, \zeta^{N})$ (related to the mean of the Nash equilibrium) is the unique solution. 
\end{proof} 

We also study convergence of the decoupling field of system \eqref{eq:FBSDE:mean}, that is, the function $u^N:[0,T]\times \R^d \rightarrow \R^d$ such that $\eta^N_t= u^N(t, m^N_t)$; then $u^N =(u^N_1,\dots, u^N_d)$ solves the following system of PDEs:
\begin{equation}
\label{eq:system_PDE}
 \begin{cases} 
    -\partial_t u_i - \frac{\sigma^2}{2N}  \Delta_{m} u_i 
    - \big(  b m - u) \cdot \nabla_{m} u_i - b u_i = \frac{\partial F_N}{\partial m_i}(m) ,  & \\
     u_i(T,m)= \frac{\partial G_N}{\partial m_i}(m). &
\end{cases}
\end{equation}

We remark that uniqueness of classical solutions to parabolic systems of PDEs is true in case of bounded coefficients, and not know for general systems with linear growth coefficients; observe that $\frac{\partial F_N}{\partial m_i}(m)$ and $\frac{\partial G_N}{\partial m_i}(m)$ have linear growth in $m$. We thus rely on the uniqueness result of Theorem \ref{lemma existence uniqueness FBSDE.mean}.

\begin{prop}
For $\sigma>0$, there exists a unique solution to system \eqref{eq:system_PDE} in the space of linear growth vectors in $C^{1,2}([0,T)\times \R^d; \R^d) \cap C([0,T]\times \R^d; \R^d)$.
\end{prop} 

\begin{proof} 
By standard results the value function of the control problem \eqref{eq:OC:N} is smooth and, thanks to the interior regularity estimates (e.g. \cite[Theorem 8.12.1]{krylov1996lectures}), its derivative is a classical solution  to system \eqref{eq:system_PDE}. Notice that under our regularity assumptions, the derivative is not $C^2$ up to the terminal time.  Uniqueness of classical solutions is then a consequence of It\^o's formula and uniqueness of \eqref{eq:FBSDE:mean}. Indeed, given a classical solution $u$, let $m$ be the solution to the forward SDE with initial point $(t_0, \nu_0)$, whereas $\eta_t$ is replaced by the function  $u(t, m_t)$; such SDE admits a unique strong solution since $\sigma>0$ and $u$ has linear growth. 
Thus, letting $\eta_t = u(t, m_t)$, the process $(\eta,m)$ solves the FBSDE  \eqref{eq:FBSDE:mean} and thus the uniqueness result Theorem \ref{lemma existence uniqueness FBSDE.mean} yields that $u(t_0, \nu_0) = \eta_{t_0}$ is uniquely determined for any $(t_0, \nu_0)$. 
\end{proof} 

\subsection{Potential mean field game}
\label{sec:potentialMFG}
Suppose the probability space $(\Omega, \mathcal F , \mathbb P)$ to be rich enough to accommodate another $d$-dimensional Brownian motion $W$ and an independent square integrable $\mathbb R^d$-valued random variable $\xi$.
Denote by $\mathcal A$ the space of square integrable $\mathbb R ^d$-valued processes which are progressively measurable with respect to the right continuous extension of the filtration generated by $W$ and $\xi$, augmented by $\mathbb P$-null sets. 

Consider the  mean field game in which, for any continuous for $m:[0, T] \rightarrow \mathbb{R}^{d}$, the representative players minimizes, over the controls $\alpha \in \mathcal A$, the cost functional
\be\label{eq:MFG.problem}
\begin{aligned}
    & J(\alpha, m)=\mathbb{E}\bigg[\int_{0}^{T} \frac{1}{2} \big( |\alpha_{t}|^{2}+ ( X_{t}  + \nabla f(m_{t}) ) ^2 \big) d t + \frac12 ( X_{T} +  \nabla g(m_{T}) ) ^2\bigg], \\
    & \text{subject to} \quad d X_{t}= (b X_t + \alpha_{t} )d t + \sigma d W_{t}, \quad X_0 = \xi.
\end{aligned}
\ee
A mean field equilibrium is a couple $(\alpha, m)$ such that $\alpha$ is optimal for $J(\cdot, m)$ over $\mathcal A$, with corresponding optimal trajectory $X$, and $m_{t}=\mathbb{E}\left[X_{t}\right]$ for any $t$. 

\begin{remark}
We warn the reader that this mean field game is potential, by considering the usual definition; see e.g. \cite[Vol I, \S 6.7.2]{CarmonaDelarue18}. Indeed, we could write the costs as functions of the measure, and show that the MFG is the flat derivative of the mean field control problem. We do not introduce these notions since they are not used in this paper. 
\end{remark}

By the stochastic maximum principle (see \cite{CarmonaDelarue18}), mean field equilibria can be characterized in terms of solutions $(X,Y,Z)$ of the related McKean-Vlasov FBSDE
\be
\label{eq:MFG.FBSDE}
\left\{\begin{array}{l}
d X_{t}= (b X_t -Y_{t}) d t + \sigma d W_{t}, \quad X_0 = \xi,  \\
d Y_{t}=- [ b^\tra Y_t + X_t + \nabla f\left(\mathbb{E}\left[X_{t}\right]\right) ] dt +Z_{t} d W_{t}, \quad 
Y_{T}=X_T + \nabla g\left(\mathbb{E}\left[X_{T}\right]\right).
\end{array}\right.
\ee 
Indeed, $(\alpha, m)$ is a mean field  equilibrium if and only if $\alpha = -Y$ for some solution $(X,Y,Z)$  the  McKean-Vlasov FBSDE \eqref{eq:MFG.FBSDE}. 
Moreover, for a solution $(X,Y,Z)$, setting $m_{t}=\mathbb{E}\left[X_{t}\right]$ and $\eta_{t}=\mathbb{E}\left[Y_{t}\right]$, by taking expectations in the system we see that $(m,\eta)$ solves the forward-backward system of ODE's
\be
\label{eq:OC.FBSDE}
\tag{OC-FBODE}
\left\{\begin{array}{l}
\dot{m}_{t}= b m_t-\eta_{t}, \quad m_0 = \mathbb E [\xi],   \\
\dot{\eta}_{t}=-[ b^\tra \eta_t + m_t + \nabla f\left(m_{t}\right)], \quad \eta_{T}=m_T+\nabla g\left(m_{T}\right) .
\end{array}\right.
\ee

Given initial conditions $(t_0,\nu_0) \in [0,T] \times \mathbb R ^d$, consider now the deterministic optimal control problem which consists in minimizing, over deterministic square integrable functions $\beta:[0,T] \to \mathbb R ^d$, the cost functional
\be
\label{eq:OC}
\tag{OC}
\begin{aligned}
    & J_c(t_0,\nu_0;\beta)=\int_{t_0}^{T}\left(\frac{1}{2}\left|\beta_{t}\right|^{2} + \frac{1}{2}\left|m_{t}\right|^{2}+ f\left(m_{t}\right)\right) d t
+ \frac{1}{2}\left|m_{T}\right|^{2} 
+  g\left(m_{T}\right), \\
&\text{subject to} \quad \dot{m}_{t}= b m_t + \beta_{t}, \quad m_{t_0} = \nu_0. 
\end{aligned}
\ee
Unless otherwise stated, we will set $J_c(\beta) = J_c(0,\nu_0;\beta)$, when $\nu_0$ is understood. 
We observe that, for $t_0=0$ and $\nu_0 = \mathbb E [\xi]$,  \eqref{eq:OC.FBSDE} is the system given by the Pontryagin maximum principle for the above control problem. 
We denote the value function by $v(t_0, \nu_0)$; i.e., set $v(t_0, \nu_0) =  \inf_\beta J_c (t_0, \nu_0; \beta)$.

The following is a classical result in deterministic control theory; see \cite{cannarsa&sinestrani2004}. 
\begin{prop}
\label{prop:optimal controls of OC}
The following hold:
\begin{itemize} 
\item[(i)] For any initial condition $(t_0,\nu_0) \in [0,T] \times \mathbb R ^d$, there exists an optimal control for \eqref{eq:OC};
\item[(ii)] If $(t_0,\nu_0) = (0,\E [\xi])$, any minimizer $(m, \beta)$ is such that $(m,\eta)$ solves \eqref{eq:OC.FBSDE} with $\beta=-\eta$;
\item[(iii)] The value function is locally Lipschitz and locally semiconcave in $\nu_0$. 
\item[(iv)] The optimal trajectory is unique for the control problem with initial condition $\nu_0$ at time $t_{0}$ if and only if $v(t_0,\cdot)$ is differentiable in $\nu_0$ - and thus for almost every $\nu_0$. \\
In this case, system \eqref{eq:OC.FBSDE} is uniquely solvable on $[t_0, T]$ and $\eta_{t_0} = \nabla_m v(t_0, \nu_0)$. 
\end{itemize}
\end{prop}

\begin{proof}
We first observe that, 
for a given (deterministic) initial condition $\nu_0$, the infimum in the control does not change if we restrict to  controls $\beta$ such that
\[
\int_0^T |\beta_t|^2 dt \leq C(1+|\nu_0|^2), 
\]
for a given constant $C$. 
This follows by the bound 
\[
\int_0^T \big(\tfrac12|\beta_t|^2 - C|\beta_t|\big) dt - C(1+|\nu_0|) 
\leq
J(\beta)\leq J(0)\leq C(1+|\nu_0|^2),
\] 
which is obtained by using the estimate 
\(
\sup_{0\leq t\leq T}|m_t| \leq |\nu_0| +C\int_0^T |\beta_t|dt 
\), and the constant zero control for the upper bound. Therefore existence of an optimal control can be proved by taking a minimizing subsequence converging in the weak topology of $\mathbb{L}^2([0,T], \R^d)$. 
This proves claim (i), 
while the proof of (ii) follows by standard arguments. 
Having a solution to \eqref{eq:OC.FBSDE}, we also get 
\[
\sup_{0\leq t\leq T}|m_t| + \sup_{0\leq t\leq T}|\eta_t|  \leq C(1+ |\nu_0|) ,  
\] 
which implies that, if $\nu_0$ belongs to a ball of radius $R$, then the control problem can be restricted to controls bounded by a constant depending on $R$. 
Thus (iii) follows by usual arguments.  
Boundedness of controls permits to apply \cite[Thm. 7.4.20]{cannarsa&sinestrani2004} which provides claim (iv) in case the value function is differentiable in $(t_0,\nu_0)$. 
The fact that differentiability of $v(t_0, \cdot)$ in $\nu_0$ implies joint differentiability of $v$ in $(t_0,\nu_0)$ is proved in \cite{cannarsa&frankowska}. 
The value functions is also shown to be locally Lipschitz and semiconcave jointly in $(t_0,\nu_0)$. 
\end{proof}

\section{$N$-player convergence result}
\label{sec:N-convergence} 
\label{section N player convergence result}
In light of Proposition \ref{prop:optimal controls of OC}, minimizes of \eqref{eq:OC} with initial condition $(t_0,\nu_0) = (0,\E [\xi])$ are mean field equilibria. 
However, there are mean field equilibria which are not minimizes of \eqref{eq:OC}, but are just stationary points.  
We show that Nash equilibria of the $N$-player game converge to minimizes of \eqref{eq:OC}, thus ruling out mean field equilibria which are not minimizes. 

To phrase our main result, consider the space of continuous functions $C ([0, T] ; \mathbb{R}^{ d} )$ endowed with the supremum norm, and the space of càdlàg functions $D([0,T]; \R^d)$ endowed with the pseudo-path topology; that is, the topology of the convergence in the measure $dt + \delta_T$, where $dt$ is the Lebesgue measure on $[0,T]$ and $\delta_T$ is the Dirac delta in $T$.

We have the following selection principle. 
\begin{thm}
\label{thm:convergence}
Let $\bar \xi ^N= \frac1N\sum_j \xi^i$ converge to $\E[\xi]$ in $L^2(\Omega)$.
\footnote{
Recall that this is equivalent to require that $\bar \xi ^N$ converge in distribution and the second moments converge.
}
If  $\sigma >0$, then  
\begin{enumerate}
\item\label{thm:convergence1} The sequence $ (m^{N}, \eta^{N} )$ is tight on $C ([0, T] ; \mathbb{R}^{ d} ) \times D([0,T]; \R^d)$, endowed with the (product) topology of uniform convergence and pseudo-path topology. 

\item\label{thm:convergence2} Any limit point $(m, \eta)$ in distribution is supported on (deterministic) minimizes of \eqref{eq:OC} with initial conditions $ (0, \E[\xi] )$.

\item\label{thm:convergence3}  The sequence $ (m^{N}, \eta^{N} )$ truly converges to the unique (deterministic) minimizer of \eqref{eq:OC} if and only if $v$ is differentiable in $ (0, \E[\xi] )$.

\item\label{thm:convergence4}
We have $\lim_N u^N (t_0, \nu_0) = \nabla_m v(t_0, \nu_0)$ at any point $t_0$ and any $\nu_0$ in which $m \mapsto v(t_0,m)$ is differentiable (thus for any $t_0$ and almost every $\nu_0$).

\end{enumerate}
\end{thm}

\begin{proof}
    We prove each claim separately. 
    \smallbreak\noindent
    \emph{Proof of \ref{thm:convergence1}.} 
    Consider the constant  $0$ control and denote by $m^0$ the related controlled state process. 
    By Theorem \ref{thm: Nash corresponds to minima of OC}, the pair $(m^N, -\eta^N)$ is  optimal for the control problem \eqref{eq:OC:N}. 
    Thus, expanding the inequality $J_c^N (-\eta^N) \leq J_c^N (0)$ and using the growth condition on $f, \nabla f, g, \nabla g$ we find
    $$
    \begin{aligned}
        &\mathbb{E} \bigg[ \int_{0}^{T}  \Big( \frac12  | \eta^N_{t}|^{2} + \frac12 |m^N_t|^2 - C |m^N_{t}|  \Big) dt  + \frac12 |m^N_T|^2 - C |m^N_{T}| \bigg] -C \\
       & \quad \leq 
        \mathbb{E} \bigg[ \int_{0}^{T}  \Big(   \frac12 |m^0_t|^2 + C |m^0_{t}|  \Big) dt  + \frac12 |m^0_T|^2 + C |m^0_{T}| \bigg] + C.
    \end{aligned}
    $$
The convergence of $\bar \xi ^N$ and classical Lipschitz estimates allow to find a constant $C$ such that
    $$
    \mathbb E \bigg[ \sup_{t \in [0,T]} |m^0_t|^2 \bigg] \leq C
    \quad \text{and} \quad
    \mathbb E \bigg[ \sup_{t \in [0,T]} |m^N_t|\bigg] \leq C \mathbb E \bigg[ \int_0^T |\eta^N_t| dt \bigg].
    $$
    Hence, the terms in the previous inequality can be rearranged as
    $$
    \begin{aligned}
     \mathbb{E} \bigg[ \int_{0}^{T}   \frac12  | \eta^N_{t}|^{2}  dt  \bigg] 
       &  \leq C 
        \mathbb{E} \bigg[ \int_{0}^{T}  \Big(   \frac12 |m^0_t|^2 +  |m^0_{t}|  + |m^N_t|  \Big) dt  + \frac12 |m^0_T|^2 +  |m^0_{T}| + |m^N_T| \bigg] 
        \\
       &  \leq C \bigg( 1 + \mathbb E \bigg[ \int_0^T |\eta^N_t| dt \bigg] \bigg),
    \end{aligned}
    $$
    which in turn implies that
    \begin{equation}
    \label{eq: a priori estimate}
    \sup_N \mathbb{E} \bigg[ \int_{0}^{T}    | \eta^N_{t}|^{2}  dt  \bigg] < \infty
     \quad \text{and} \quad
     \sup_N  \mathbb E \bigg[ \sup_{t \in [0,T]} |m^N_t|^2 \bigg] < \infty,
    \end{equation}
    where the second estimate follows from the first one. 
    
    Next, we want to show that $(m^N)_N$ is tight on $C([0,T], \R^d)$. 
    Let $h>0$, and let $\mathcal T_h$ be the set of $[0,T]$--valued stopping times $\tau$ such that $ \tau +h\le T$.
  For every $\tau\in \mathcal T_h$, it follows by H\"older's inequality  that 
  $$
  \begin{aligned}
    \E[|m^N_{\tau+h} - m^N_\tau|] 
    &\le \E\bigg[\int_{\tau}^{\tau+h}|b m^N_s  - \eta^N_s | ds \bigg] + \frac1{\sqrt N} |\sigma| \E[|\bar W_{\tau+h}^N - \bar  W_{\tau}^N|]\\
    &\le\sqrt h \E\bigg[\int_{\tau}^{\tau+h}(|b|^2 |m^N_s|^2  +| \eta^N_s|^2 ) ds \bigg] ^{\frac12} + |\sigma|\sqrt{h}\\
    & \le C\sqrt{h} , 
  \end{aligned}
  $$
  where the last inequality follows from \eqref{eq: a priori estimate}.
  It follows by Aldous' tightness criterion \cite[Theorem 16.9]{billingsley1999convergence} that the sequence $(m^N)_N$ is tight on $D([0,T]; \R^d)$ with the Skorohod $J_1$ topology, and thus also on $C([0,T]; \R^d)$ with the topology of uniform convergence.

In order to show the tightness of the controls $(\eta^N)_N$, we employ the FBSDE and the Meyer-Zheng tightness criterion. 
Let $\Pi$ be the set of partitions  $\pi  =\{0=t_0, t_1,\dots, t_{n_\pi} = T\}$ of $n_\pi \in \mathbb N$ points $t_i \in [0,T]$ with $t_i < t_{i+1}$.
For $\pi \in \Pi$, set 
$$
V_\pi(\eta^N):= \sum_{i=0}^{n_\pi-1} \big| \mathbb E [ \eta^N_{t_{i+1}} -  \eta^N_{t_i}| \mathcal F_{t_i} ] \big| + \mathbb E [ |\eta_T^N| ].
$$ 
The FBSDE representation of $\eta^N$ as in \eqref{eq:FBSDE:mean} (with $F_N, G_N$ as in \eqref{eq:FBSDE.mean.niceFG}) gives 
$$
\begin{aligned}
V_\pi(\eta^N) & = \sum_{i=0}^{n_\pi-1} \bigg| \mathbb E \bigg[  \int_{t_i}^{t_{i+1}}[ b^\tra \eta^N_t + \nabla F_N (m_{t}^{N}) ] dt +\sum_{j} \zeta_{t}^{j,N} d W_{t}^{j}     \bigg| \mathcal F_{t_i} \bigg] \bigg| + \mathbb E[ | \nabla G_N (m_T^N)|],
\end{aligned}
$$
and using the square integrability of $\zeta^{j,N}$ and the growth conditions on $f,g$, we find
$$
\begin{aligned}
\E [V_\pi(\eta^N) ] & \leq \mathbb E \bigg[     \int_{0}^{T} \big| b^\tra \eta^N_t + \nabla F_N (m_{t}^{N})\big| dt   + | \nabla G_N (m_T^N)| \bigg] \\
& \leq  C \bigg( 1 + \mathbb{E} \bigg[ \int_{0}^{T}    | \eta^N_{t}| dt  + \sup_{t \in [0,T]} |m^N_t| \bigg] \bigg).
\end{aligned}
$$
Thus, by using  the estimates in \eqref{eq: a priori estimate}, we obtain
$$
\sup_N \sup_{\pi \in \Pi} \E [V_\pi(\eta^N) ] 
\leq C \bigg( 1 + \mathbb{E} \bigg[ \int_{0}^{T}    | \eta^N_{t}|^{2} dt  + \sup_{t \in [0,T]} |m^N_t|^2 \bigg] \bigg)< \infty. 
$$
The Meyer-Zheng tightness criterion (see \cite[Theorem 3.9, Volume II]{CarmonaDelarue18}) implies that the sequence $(\eta^N)_N$ is tight on the space of càdlàg function $D ([0,T]; \mathbb R ^d )$, endowed with the convergence in the measure $dt + \delta _T$, where $\delta_T$ denotes the Dirac's delta in $T$. 

Therefore, we conclude that the sequence $(m^N,\eta^N)$ is tight in $C ([0,T]; \mathbb R ^d ) \times D ([0,T]; \mathbb R ^d )$ endowed with the product topology. 

\smallbreak\noindent
\emph{Proof of \ref{thm:convergence2}.} 
Let the process $(m,\eta)$ be a limit point in distribution of the sequence $(m^N,\eta^N)$.
By Skorokhod representation theorem, we can find a probability space $(\tilde \Omega, \tilde{\mathcal F}, \tilde{\mathbb P})$ supporting 
$\mathbb R ^d \times C ([0,T]; \mathbb R ^d )  \times C ([0,T]; \mathbb R ^d )   \times  D ([0,T]; \mathbb R ^d )$-valued r.v.'s
$
(\tilde \xi ^N, \tilde W ^N, \tilde m ^N, \tilde \eta ^N) 
$, $N \geq 1$, 
$
(\tilde \xi , \tilde W , \tilde m , \tilde \eta ) , 
$
such that: 
\begin{equation}\label{eq Skorokhod limits}
\begin{aligned}
    &\text{$\mathbb P \circ (\bar \xi ^N, \bar W ^N,  m ^N,  \eta ^N) ^{-1} =  \tilde{\mathbb P} \circ (\tilde \xi ^N, \tilde W ^N, \tilde m ^N, \tilde \eta ^N) ^{-1}$ for any $N \geq 1$;}\\
    &\text{$\mathbb P ^* =  \tilde{\mathbb P} \circ (\tilde \xi , \tilde W , \tilde m , \tilde \eta ) ^{-1}$;}\\
    &\text{$
(\tilde \xi ^N, \tilde W ^N, \tilde m ^N, \tilde \eta ^N)
\to
(\tilde \xi , \tilde W , \tilde m , \tilde \eta ) 
$, $\tilde{\mathbb P}$-a.s.,} 
\end{aligned}
\end{equation}
where the convergence is meant in $\mathbb R ^d \times C ([0,T]; \mathbb R ^d )  \times C ([0,T]; \mathbb R ^d )   \times  D ([0,T]; \mathbb R ^d )$. 
We denote with $\tilde\E$ the expectation under the probability measure $\tilde{\mathbb P}$.

We first recall that by assumption $\bar \xi^N$ converges to $\E[ \xi ]$ and observe that $\bar W ^N$ converges to $0$. 
Hence we have
$\tilde \xi  = \E [\xi]$ and $\tilde W = 0$.
Secondly, the convergence of  $\tilde \eta^N$ in the measure $dt+\delta_T$ together with the a priori estimates in \eqref{eq: a priori estimate} implies that
$$
\lim_N \tilde\E \bigg[ \int_0^T |\tilde \eta^N_t- \tilde \eta_t| dt \bigg] =0.
$$
Moreover, using that $m^N$ solves the SDE controlled by $-\eta^N$, we obtain that $\tilde m^N$ solves the SDE controlled by $- \tilde \eta^N$ with noises $\tilde \xi^N$ and $\tilde W ^N$.
Thus, using  Gr\"onwall estimates and the previous limit, we obtain the convergence
\[
\lim_N \tilde\E \bigg[ \sup_{t\in [0,T]} |\tilde m^N_t-\tilde m_t|\bigg] =0,
\]
and that
$$
\tilde m _ t  = \tilde \xi + \int _0^t (b \tilde m_s - \tilde \eta_s) ds. 
$$

We now show that 
\[
\tilde\E[J_c(\tilde \eta)] \leq \liminf_N J^N_c(\eta^N),
\] 
where $\tilde\E[J_c( \tilde \eta)]$ denotes the expectation of the deterministic cost in \eqref{eq:OC} in which the control $\eta$ is random.
This is a consequence of the above convergence, continuity and linear growth of $f$ and $g$, and by Fatou's lemma which yields 
\begin{align*} 
\tilde\E \bigg[ \int_0^T |\tilde{\eta}_t|^2 dt \bigg] &\leq \liminf_N \tilde\E \bigg[ \int_0^T |\tilde \eta^N_t|^2 dt\bigg], \\
\tilde\E\bigg[\int_0^T |\tilde{m}_t|^2 dt \bigg] &\leq \liminf_N \tilde\E \bigg[\int_0^T |\tilde m^N_t|^2 dt\bigg], \qquad 
\tilde\E \big[|\tilde m_T|^2\big] \leq \liminf_N \tilde\E \big[|\tilde m^N_T|^2\big].
\end{align*}
Next, for any given deterministic control $\beta$, denote by $\tilde{m}^{\beta,N}$ the solution to SDE in \eqref{eq:OC:N} in the probability space $(\tilde \Omega, \tilde{\mathcal F}, \tilde{\mathbb P}, \tilde W^N)$, and by $\tilde m^\beta$ the solution to the ODE in \eqref{eq:OC}; that is, 
$$
\begin{aligned}
 d \tilde m^{\beta,N}_t &= (b \tilde m^{\beta,N}_t +  \beta) dt +  \frac{\sigma}{\sqrt N} d \tilde W ^N_t, \quad \tilde m^{\beta,N}_0 = \tilde \xi^N, \\
 d \tilde m^{\beta}_t &= (b \tilde m^{\beta}_t +  \beta) dt , \quad \tilde m^{\beta}_0 = \tilde \xi. 
\end{aligned}
$$
The assumed convergence on the initial conditions imply that 
\[
\lim_N \tilde \E \bigg[ \sup_{t\in[0,T]}|\tilde m^{\beta, N}_t -\tilde m^\beta_t|^2 \bigg] =0, 
\] 
and thus $\lim_N J^N_c(\beta) = J_c(\beta)$. 
Hence the optimality of $\eta^N$ yields
\[
\tilde\E[J_c(\tilde \eta)] \leq \liminf_N J^N_c(\eta^N) \leq \lim_N J^N_c(\beta) = J_c(\beta), 
\] 
for any deterministic control $\beta$, which in turn implies that the measure  $\tilde{\mathbb P} \circ (\tilde \eta) ^{-1}$ charges only the minimizers of the map $J_c:  L^2 ([0,T]; \mathbb R ^d) \to \mathbb R$.

\smallbreak\noindent
\emph{Proof of \ref{thm:convergence3}}.  Claim follows directly from Claim \ref{thm:convergence2} and from Proposition \ref{prop:optimal controls of OC}.

\smallbreak\noindent
\emph{Proof of \ref{thm:convergence4}}. 
For any initial point $(t_0, \nu_0)$, we have $\eta^N_{t_0} =  u^N(t_0, \nu_0)$, where $\eta^N$ solves \eqref{eq:FBSDE:mean} on $[t_0,T]$, by the stochastic Pontryagin maximum principle, noticing that $u^N$ is the gradient (in $m$) of the value function of \eqref{eq:OC:N}.  
Let $t_0$ and $\nu_0$ such that $m \mapsto v(t_0,m)$ is differentiable in $\nu_0$. 
Thanks to Proposition 2 (iv), we have just to show that $\lim_N \eta^N_{t_0} = \eta_{t_0}$. In the above setup, we have 
$$\lim_N \tilde \E\bigg[ \sup_{t\in [t_0, T]} |\tilde m^N_t -\tilde m_t | \bigg],$$
but convergence is not in $L^2$. We apply It\^o-Tanaka formula to $\eta^N-\eta$ to get 
\begin{align*}
    \frac{d}{dt} \tilde{\E} [|\tilde \eta^N_t - \tilde \eta_t|] &\geq 
-  b \tilde{\E} [|\tilde \eta^N_t - \tilde \eta_t|]  - C \tilde \E [|\tilde m^N_t - \tilde m_t|]    - \frac{C}{N} (1+ \tilde \E [|m^N_t|] )
    \\
    \E [|\tilde \eta^N_T - \tilde \eta_T|] &\leq C \tilde \E[ |\tilde m^N_T - \tilde m_T|]    - \frac{C}{N} (1+ \tilde \E [|m^N_T|] ) ,
\end{align*}
and thus Gronwall's inequality yields
\[
\sup_{t\in [t_0, T]} \tilde{\E} [|\tilde \eta^N_t - \tilde \eta_t|] \leq C\Big( 
\frac1N + \sup_{t\in [0,T]} \tilde \E [|\tilde m^N_t - \tilde m_t|] \Big).  
\] 
Hence we obtain $\lim_N \sup_{t\in [t_0, T]} \tilde{\E} [|\tilde \eta^N_t - \tilde \eta_t|] =0$, which gives the claim. 
\end{proof}

\section{Mean field game with common noise}
\label{sec:MFG-common-noise}
We add a common noise to the linear-quadratic potential mean field game. 
Suppose the probability space $(\Omega, \mathcal F , \mathbb P)$ to be rich enough to accommodate another $d$-dimensional $ (\xi, W)$-independent Brownian motion $B$: the common noise.
Denote by $\mathcal A _\varepsilon$ the space of square integrable $\mathbb R ^d$-valued processes which are progressively measurable with respect to the right continuous extension of the filtration generated by $W, B$ and $\xi$, augmented by $\mathbb P$-null sets.
Denote by $\mathbb F ^B = (\mathcal F _t^B)_{t \in [0,T]}$ the filtration generated by $B$, augmented by $\mathbb P$-null sets.

Let $\epsilon>0$ and consider, for any control $\alpha^\varepsilon \in \mathcal A _\varepsilon$, the dynamics
$$
d X_{t}^\epsilon= (b X_t^\epsilon + \alpha_{t}^\epsilon )d t + \sigma d W_{t}
+\epsilon dB_t, \quad X_0^\epsilon = \xi,
$$
and the cost
$$
J^\varepsilon (\alpha^\epsilon, m)=\mathbb{E}\bigg[\int_{0}^{T} \frac{1}{2} \big( |\alpha_{t}^\epsilon|^{2}+ ( X_{t}^\epsilon  + \nabla f(m_{t}) ) ^2 \big) d t +  \frac12 ( X_{T}^\epsilon +  \nabla g(m_{T}) ) ^2\bigg],
$$
for a generic square integrable $\mathbb F ^B$-progressively measurable process $m:[0, T] \rightarrow \mathbb{R}^{d}$.
In a mean field equilibrium $(\alpha^\epsilon, m^\epsilon)$: $\alpha^\epsilon$ is optimal for $m^\epsilon$ over the controls in $\mathcal A _\varepsilon$, with corresponding optimal trajectory $X^\epsilon$, and it holds $m_{t}^\epsilon=\mathbb{E}[X_{t}^\epsilon | \mathcal{F}_t^{B}]$ for any $t$, $\P$-a.s.; this is a strong formulation of the MFG with common noise, see \cite[Vol I, \S 2.2]{CarmonaDelarue18}. 

Mean field equilibria satisfy the FBSDE
\be
\label{eq:MFG_eps_FBSDE}
\left\{\begin{array}{l}
d X_{t}^\epsilon= (b X_t^\epsilon -Y_{t}^\epsilon) d t + \sigma d W_{t} + \epsilon dB_t, \quad X_0^\epsilon = \xi,  \\
d Y_{t}^\epsilon=- [ b^\tra Y_t^\epsilon + X_t^\epsilon + \nabla f(\mathbb{E}[X_{t}^\epsilon| \mathcal{F}_t^{B}]) ] dt +Z_{t}^\epsilon d W_{t} + \tilde Z^\epsilon_t dB_t, \quad 
Y_{T}^\epsilon=X_T + \nabla g(\mathbb{E}[X_{T}^\epsilon| \mathcal{F}_t^{B}]).
\end{array}\right.
\ee 
Taking conditional expectations, the processes $m_{t}^\epsilon=\mathbb{E}[X_{t}^\epsilon |\mathcal{F}_t^B]$,  $\eta_{t}^\epsilon=\mathbb{E}[Y_{t}^\epsilon|\mathcal{F}_t^B]$ and $\zeta_{t}^\epsilon=\mathbb{E}[\tilde Z_{t}^\epsilon|\mathcal{F}_t^B]$ solve the FBSDE
\be
\label{eq:FBSDE:MFGeps}
\left\{\begin{array}{l}
d{m}_{t}^\epsilon= (b m_t^\epsilon-\eta_{t}^\epsilon) dt +\epsilon dB_t, \quad m_0^\epsilon = \mathbb E [\xi],   \\
d{\eta}_{t}^\epsilon=-[ b^\tra \eta_t^\epsilon + m_t^\epsilon + \nabla f\left(m_{t}^\epsilon\right)]dt + \zeta_t^\epsilon dB_t, \quad \eta_{T}^\epsilon=m_T^\epsilon+\nabla g\left(m_{T}^\epsilon\right) .
\end{array}\right.
\ee
By Pontryagin's maximum principle, the above systems can be seen as the optimality condition of the stochastic control problem: 
\be
\label{eq:OC_eps}
\tag{OC$_\epsilon$}
\begin{aligned}
&J_c^\epsilon(\beta)=\int_{0}^{T}\left(\frac{1}{2}\left|\beta_{t}\right|^{2} + \frac{1}{2}\left|m_{t}\right|^{2}+ f\left(m_{t}\right)\right) d t
+ \frac{1}{2}\left|m_{T}\right|^{2} 
+  g\left(m_{T}\right), \\
&\text{subject to} \quad d{m}_{t}= (b m_t + \beta_{t})dt + \epsilon dB_t, \quad m_0 = \E [\xi], 
\end{aligned}
\ee
where the controls are square integrable and $\mathbb F ^B$-progressively measurable.

With the help of Theorem \ref{lemma existence uniqueness FBSDE.mean}, we immediately get the following result.
\begin{prop}
The FBSDE \eqref{eq:MFG_eps_FBSDE} is uniquely solvable and its solution $(X^\epsilon,Y^\epsilon,Z^\epsilon,\tilde{Z}^\epsilon)$ gives the unique solution $(\alpha^\epsilon,  m^\epsilon) =  ( -Y^\epsilon, \mathbb{E}[X^\epsilon |\mathcal{F}^B])$ 
to the MFG with common noise. 
The control problem \eqref{eq:OC_eps} has a unique optimal control $-\eta^\epsilon$, and the related optimal trajectory is  $m^\epsilon$. 
In particular, we have $m^\epsilon_t= \E[X^\epsilon_t | \mathcal{F}^B_t]$ and $\eta^\epsilon_t= \E[Y^\epsilon_t | \mathcal{F}^B_t]$, for any $t$, $\mathbb{P}$-a.s.
\end{prop}

Let $v^\epsilon$ be the value function of \eqref{eq:OC_eps}. We obtain the following convergence result, whose proof is similar to that of Theorem \ref{thm:convergence}.  
\begin{thm} 
\label{thm:eps_convergence}
The following hold:
\begin{enumerate}
\item\label{thm:eps_convergence1} The family $ (m^{\epsilon}, \eta^{\epsilon} )$ is tight on $C ([0, T] ; \mathbb{R}^{ d} ) \times D([0,T]; \R^d)$, endowed with the (product) topology of uniform convergence and pseudo-path topology.

\item\label{thm:eps_convergence2} As $\epsilon \to 0$, any limit point $(m, \eta)$ in distribution is supported on (deterministic) minimizes of \eqref{eq:OC}.

\item\label{thm:eps_convergence3}  The sequence $ (m^{\epsilon}, \eta^{\epsilon} )$ converges, as $\epsilon \to 0$, to the unique (deterministic) minimizer of \eqref{eq:OC} if and only if $v$ is differentiable in $ (0, m_{0} )$. 

\item We have $\lim_\epsilon \nabla_m v^\epsilon (t_0, \nu_0) = \nabla_m v(t_0, \nu_0)$ at any point $t_0$ and any $\nu_0$ in which $m \mapsto v(t_0,m)$ is differentiable (thus for any $t_0$ and almost every $\nu_0$).
\end{enumerate}
\end{thm}

\section{Examples of selection}
\label{sec:examples}
In this section we discuss several examples which illustrate how to derive properties of the limiting measure from the previous result.
In the whole section, we assume that 
$
\bar \xi^N$ converges to $ m_0 = \E [\xi].
$

Our first example concerns the case of a unique minimizer. 
\begin{proposition}
If $\sigma >0$ and the function $J_c$ has only one optimal control $\hat \beta$ with optimal trajectory $\hat m$ solving $\hat m _t = m_{0}+\int_0^t (b\hat m_s + \hat \beta_s) ds$, 
then  the sequence ${m}^{N}$ converges to $\hat m$.
\end{proposition}
\begin{proof}
    Obvious from Theorem \ref{thm:convergence}.
\end{proof}

\begin{example}
    There are two typical examples in which the assumptions of the previous proposition are verified:
    \begin{enumerate}
        \item  The case is which the value function in differentiable in $(0,m_0)$, as in Theorem \ref{thm:convergence}-\ref{thm:convergence3}.
        \item The case in which the functions $f,g$ are convex. 
    In this case the Lasry-Lions monotonicity condition is satisfied and the mean field equilibrium is unique, so that our proposition recovers \cite[Corollary 2.10]{lacker.2016general}.
    \end{enumerate}  
\end{example}

\subsection{One-dimensional symmetric selection}
\label{sec:1D_example}

In dimension 1, we discuss a selection principle with two minimizers, thus in which multiple mean field equilibria could arise.
We state the following proposition, which assumes the existence of two distinct optimal controls (more explicit conditions are given in the next subsection). 
\begin{theorem}
    \label{prop case two minima general}
    Consider the case  $\sigma >0$, $d=1$ with initial condition such that $\mathbb  P \circ (-\bar \xi ^N)^{-1} = \mathbb  P \circ (\bar \xi ^N)^{-1}$ and $\nu_0=0 $. 
    Suppose
    \begin{enumerate}
        \item $f,g$ to be even (i.e., $f(-m)= f(m)$ and $g(-m)= g(m)$ for any $ m  \in \R$); 
        \item $J_c$ to have two and only two (distinct) optimal controls, one of which strictly positive.
    \end{enumerate} 
    Then, if $\hat \beta$ is the positive optimal control, the control $- \hat \beta$ is optimal and we have 
    \[
    \lim_{N\to\infty} \P \circ (\hat{m}^{N} )^{-1} = \frac12 \delta_{m^-} +  \frac12 \delta_{ m^+},
    \]
    where $m^+_t:= \int_0^t (b\hat m_s + \hat \beta_s) ds $ and $m_t^- :=  \int_0^t(b\hat m_s - \hat \beta_s) ds =-m^+_t$.
\end{theorem}

\begin{example}\label{example delarue}
The assumption on the optimal controls of $J_c$ in the previous proposition is satisfied in the example  studied in \cite{DelarueFT19}, where  $f(m)=0$ and $g$ is defined as
$$
g(m):= - \frac{m}{r_\delta} \mathbbm{1}_{\{ |m|\leq r_\delta \} } - \text{sign}(m) \mathbbm{1}_{\{ |m| > r_\delta \} }, \ m\in \mathbb{R}, \quad \delta \in (0,T), \quad  r_\delta:= \int_\delta^T {w_s^{-2}} ds, 
$$
with $w_t:=\exp\big[\int_t^T( - b + \eta_s) ds\big]$, $\eta$ solution to the Riccati equation
$
\frac{d \eta_t}{dt} = \eta_t^2 - 2 b \eta_t - 1$, $\eta_T = 1$.
In this case, the two optimal trajectories are given by $m^-_t = -  w_t \int_0^t {w_s^{-2}} ds$ and $m^+_t =  w_t \int_0^t {w_s^{-2}} ds$, and Theorem \ref{prop case two minima general} recovers the $N$-player selection of \cite[Theorem 4]{DelarueFT19},
with the difference that \cite{DelarueFT19} consider $N$ player games in which player $i$ optimization problem depends only on $\frac{1}{N-1} \sum _{j\ne i} X_T^j$.
\end{example} 
\begin{remark}
Notice that in Example \ref{example delarue} the game is submodular and the two equilibrium-trajectories $m^-,m^+$ correspond to the minimal and maximal equilibria.
The game also admits an intermediate the equilibrium $(\alpha^0,m^0)=(0,0)$ which is not selected as the limit of the empirical measure of the Nash equilibria.
\end{remark}

\begin{proof}[Proof of Theorem \ref{prop case two minima general}]
Since  $f,g$ are even, $f',g'$ are odd (i.e., $f(-m)= - f(m)$ and $g(-m)= - g(m)$ for any $ m  \in \R$) and $m f'(m), m g'(m)$ are even as well. 
Thus, computing the derivatives in $m$ of the costs of the control problem  \eqref{eq:OC:N}, we find
$$
\begin{aligned}
 F_N'( - m)&=
\Big( \I_d + \frac{1}{N}  f''(m) \Big) \Big( m +   f' (m) \Big) = - F_N'(  m), \\
 G_N'(-m)&=
\Big( \I_d + \frac{1}{N} g''(m) \Big) \Big( m +  g' (m) \Big)= - G_N'(m).
\end{aligned}
$$
By the stochastic maximum principle, any optimal control of the control problem \eqref{eq:OC:N}, writes as $-\eta^N$ for some solution $(m^N,\eta^N, \zeta^{N})$ to  the FBSDE \eqref{eq:FBSDE.mean.niceFG}. 
By the previous identities, the process $(\tilde m^N,\tilde \eta^N,\tilde \zeta^{N}) = (-m^N,-\eta^N, \zeta^{N})$ solves the FBSDE
$$
 \left\{\begin{array}{l} 
 d \tilde m ^N_{t} = (b \tilde m^N_t - \tilde \eta^N_{t}) d t- \frac{\sigma}N  \sum_{j} d W_{t}^{j}
 \quad \tilde m^N_{0}= - \bar \xi ^N, \\
d \tilde \eta^N_{t} =-[ b \tilde \eta^N_t +  F_N' ( \tilde m ^N_{t} ) ] dt - \sum_{j} \tilde\zeta_{t}^{jN} d W_{t}^{j}, \quad \eta_{T}= G_N' (\tilde m^N_{T}),
\end{array}\right.
$$
which the same system as \eqref{eq:FBSDE.mean.niceFG}, but driven by the noises $(-\bar \xi ^N, - \bm{ W})$.
By our assumption,  $\mathbb  P \circ (-\bar \xi ^N, -  \bm{ W})^{-1} = \mathbb  P \circ (\bar \xi ^N,   \bm{ W})^{-1}$,
and by weak uniqueness of the solution to the FBSDE \eqref{eq:FBSDE.mean.niceFG} (following from Theorem \ref{lemma existence uniqueness FBSDE.mean}), we deduce that $\mathbb  P \circ (m^N)^{-1} = \mathbb  P \circ (- m^N)^{-1}$. 
Hence we have 
$
\mathbb P [m^N_t <0] = \mathbb P [m^N_t >0] \leq \frac12 
$
for any $t$,
and since $\mathbb P [m^N_t =0] = 0$, we obtain that
\be\label{eq 12 NE}
\mathbb P [m^N_t <0] = \mathbb P [m^N_t >0] = \frac12.
\ee

Let $\mu \in \mathcal P (C([0,T]; \R))$ be the limit in distribution of a subsequence (not relabeled) of $m^N$ converging in distribution on $C([0,T], \R)$.
By Theorem \ref{thm:convergence}, the limit $\mu$ is supported only on the optimal trajectories of $J_c$, which are  $m^+_t:= \int_0^t  (bm^+_t + \hat \beta_s) ds $ and $m_t^- = -m^+_t$. 
Thus, there exists $p \in [0,1]$ such that
$$
\mu = p \delta_{m^-} +  (1-p) \delta_{ m^+},
$$
where $\delta_{(x_t)_t}$ denoted the Dirac delta on a particular trajectory $(x_t)_t \in C([0,T]; \R)$.
In order to determine $p$, we use \eqref{eq 12 NE}.
Indeed, denoting with $\pi_T$ the map $\pi_T: C([0,T]; \R) \to \R $, $\pi_T( (x_t)_t) := x_T$, the continuous mapping theorem gives that $\mathbb P \circ (m^N_T)^{-1}$ converges in distribution to $\mu \circ (\pi_T)^{-1} =  p \delta_{m^-_T} + (1-p) \delta_{m^+_T} $.
Notice that, since $m_T^- <0<m^+_T$, we have
$$
\begin{aligned}
    \mu \circ (\pi_T)^{-1} (-\infty,0) &= (p \delta_{m^-_T} +  (1-p) \delta_{ m^+_T})(-\infty,0) = p,  \\
    \mu \circ (\pi_T)^{-1} (0,\infty)& = (p \delta_{m^-_T} +  (1-p) \delta_{ m^+_T})(0,\infty)=  1-p .
\end{aligned}
$$
Hence, by Portmanteau theorem and \eqref{eq 12 NE}, we deduce that
$$
\begin{aligned}
   p = \mu \circ (\pi_T)^{-1} (-\infty,0) &\leq \liminf_N \mathbb P [m^N_T <0] =  \frac12, \\
    1-p = \mu \circ (\pi_T)^{-1} (0,\infty) &\leq \liminf_N \mathbb P [m^N_T >0] =  \frac12,
\end{aligned}
$$
so that $p=\frac12$ and $\mu =  \frac12 \delta_{m^-} +  \frac12 \delta_{ m^+}$.
Finally, since every subsequence have the same limit $\mu$, the sequence $(m^N)_N$ converges to $\mu$, thus completing the proof.
\end{proof} 

\subsection{The case of constant optimal controls}
\label{sec:constant}
In this subsection we will focus on the case in which the representative player minimization problem is
\be\label{eq:MFG.simple}
\begin{aligned}
    & J(\alpha, m)=\mathbb{E}\bigg[\int_{0}^{T} \frac{1}{2}  |\alpha_{t}|^{2} d t + \frac12 ( X_{T} +  \nabla g(m_{T}) ) ^2\bigg], \\
    & \text{subject to} \quad d X_{t}= \alpha_{t} d t + \sigma d W_{t}, \quad X_0 = \xi,
\end{aligned}
\ee
and the related optimal control problem is
\be
\label{eq:OC.simple}
\begin{aligned}
    & J_c(t_0,\nu_0;\beta)=\int_{t_0}^{T} \frac{1}{2}\left|\beta_{t}\right|^{2}  d t
+ \frac{1}{2}\left|m_{T}\right|^{2} 
+  g\left(m_{T}\right), \\
&\text{subject to} \quad \dot{m}_{t}=  \beta_{t}, \quad m_{t_0} = \nu_0. 
\end{aligned}
\ee
For this model, all the results of the previous sections applies with the same arguments. 

We first discuss the following property of optimal controls.
\begin{lemma}
\label{lemma constant optimal controls}
If $\hat{\beta}$ is optimal for $J_c$ as in \eqref{eq:OC.simple}, then $\hat{\beta}$ is constant in time; i.e., $ \hat{\beta}_{t}=\hat{a}$ for any $ t \in [0,T]$.
\end{lemma}

\begin{proof} 
By the Pontryagin maximum principle, the optimal control $\hat \beta$ equals the backward component $\eta$ of some solution $(m,\eta)$ to the  forward-backward system of ODE's
\be
\label{eq:OC.FBSDE.simple}
\left\{\begin{array}{l}
\dot{m}_{t}= -\eta_{t}, \quad m_{t_0} = \nu_0,   \\
\dot{\eta}_{t}=0, \quad \eta_{T}=m_T+\nabla g\left(m_{T}\right) .
\end{array}\right.
\ee
The proof is completed by noticing that every solution has backward component $\eta$ constant in time.
\end{proof} 

Lemma \ref{lemma constant optimal controls} implies that the optimization problem is equivalent to the static optimization problem of the function
$$
a \mapsto U\left(t_0, \nu _0, a \right) :=  \frac{T-t_0}{2}|a|^{2}+G\left(\nu_{0}+(T-t_0) a \right) , 
$$
where we consider the function $G(y) = \frac{1}{2}\left|y\right|^{2} 
+  g\left(y\right),$ $y \in\mathbb R^d$.

For simplicity, we discuss the case $T =1$.
\begin{proposition}
\label{prop condition g}
    Assume  $\sigma >0$, $d=1$, $g$ to be concave, and such that there exist $C_-, C_+$ such that $g''(m) + 2 <0$ for $m \in (C_-,C_+)$ and $g''(m) + 2 \geq0$ for $m \in (-\infty,C_-] \cup [C_+, \infty)$. 
    Then there exist at most two solutions of the optimal control problem and one solution of the optimal control problem is selected with positive probability. 
\end{proposition}

\begin{proof}
The second order derivative (in the variable $a$) of $U(0,m_0;\cdot)$ is $U''(0,m_0;a) = 2 + g''(m_0 + a)$.  
Studying the convexity/concavity of $U(0,m_0;\cdot)$, one deduces that $U(0,m_0;\cdot)$ admits at most two (distinct) minima, which by Lemma \ref{lemma constant optimal controls} correspond to two distinct optimal controls for \eqref{eq:OC.simple}. 
The rest of the claim follows by Theorem \ref{thm:convergence}. 
\end{proof}

In particular, we obtain the following example of a submodular mean field game (see \cite{dianetti.ferrari.fischer.nendel.2019}) in which only extremal equilibria are plausible limits of the $N$-player game. 
\begin{proposition}
    \label{prop symmetric case two minima}
    Consider the case $\sigma >0$, $d=1$ with initial condition $\nu_0=0 $. 
    Suppose $g $ to be concave, even (i.e. $g(-m)= g(m)$)  and such that there exist $C_+$ such that $g''(m) + 2 <0$ for $m \in [0,C_+)$ and $g''(m) + 2 \geq 0$ for $m \in [C_+, \infty)$.
    Then, $U(0,0; \cdot)$ has two and only two (distinct) minima $\bar a >0 $ and $- \bar a$ and we have
    \[
    \lim_{N\to\infty} \P \circ (\hat{m}^{N} )^{-1} = \frac12 \delta_{m^-} +  \frac12 \delta_{ m^+},
    \]
    where $m_t^- := - t \bar{a}$ and $m^+_t:=t \bar{a}$.
    Moreover, $(\alpha^0,m^0)=(0,0)$ is a mean field equilibrium which is not selected as limit of Nash equilibria. 
\end{proposition}

\begin{proof}
The characterization of the limit points follows by using Lemma \ref{lemma constant optimal controls}, Proposition \ref{prop condition g} and Theorem \ref{prop case two minima general}, after noticing that the minimum points of $U(0,0,; \cdot)$ are necessarily different from $0$.
Notice that $g'(0)=0$, so  that $(\sigma W,0)$ is a solution to the system  
$$
\left\{\begin{array}{l}
d X_t = -Y_{t} dt + \sigma dW_t, \quad \mathbb E[X_0] = 0,   \\
d Y_{t}= Z_t dW_t, \quad Y_{T}=\mathbb E[X_T]+  g' \left(  \mathbb E[X_T] \right) .
\end{array}\right.
$$
Thus, $(0,0) = (0,\mathbb E [\sigma W])$ is a mean field equilibrium by the Pontryagin maximum principle.  
Such an equilibrium does not correspond to any minimum, hence it is not selected as limit of Nash equilibria. 
\end{proof}



\subsection{Multi-dimensional symmetric selection} 
\label{sec:multiD}
We finally provide a multidimensional example with an infinite number of mean field equilibria. 
\begin{theorem}
\label{prop multidimensional} 
    Assume $\sigma >0$, $\nu_0=0 \in \mathbb R ^d$,  the distribution of $\bar \xi^N$ to be invariant under rotations,   $g(m)=\tilde g(|m|)$ for a concave function $\tilde g \in C^2([0,\infty))$ with $\tilde g '(0)=0$, and such that there exists $ C_+$ such that $\tilde g''(m) + 2 <0$ for $m \in [0,C_+)$ and $\tilde g''(m) + 2 \geq0$ for $m \in [C_+, \infty)$. 
    Then 
    \[
    \lim_{N\to\infty} \P \circ (\hat{m}^{N} )^{-1} =  \P \circ (m^\Theta)^{-1}, 
    \]
    where $m^\Theta_t = \Theta \hat a \,t$, $\hat a$ is the (unique) minimum of the function $a_1 \mapsto U (0,0; (a_1,0,...,0))$ on $(0,\infty)$, and $\Theta$ is a uniform random variable on the unit sphere $S^{d-1} = \{ r \in \R^d \, | \, |r|=1 \}$. 
    
    Moreover, $(\alpha^0,m^0)=(0,0)$ is a mean field equilibrium which is not selected as limit of Nash equilibria. 
\end{theorem}

\begin{proof}
We first notice that the unique solution $(m^{N}, \eta^{N}, \zeta^{N})$ of the FBSDE
$$
\left\{\begin{array}{l} 
d m_{t}^{N}= -\eta_{t}^{N}  dt+\frac{\sigma}{N} \sum_{j} d W_{t}^{j},  \quad m^N_{0}=\bar \xi^{N}, \\
d \eta_{t}^{N}= \sum_{j} \zeta_{t}^{j,N} d W_{t}^{j}, \quad  \eta_{T}^{N}=\Big( \I_d + \frac{1}{N} \nabla^{2} g(m^N_T) \Big) \Big( m^N_T +  \nabla g (m^N_T) \Big), 
\end{array}\right.
$$
which gives the optimal control for the control problem \eqref{eq:OC.simple},
is invariant under rotations.
Indeed, for  a generic rotation  $R$ of  $\R^d$,
with some basic calculus we see that
the process $(R m^{N}, R \eta^{N},  \zeta^{N})$ solves the system
$$
\left\{\begin{array}{l} 
d R m_{t}^{N}= - R \eta_{t}^{N}  dt +\frac{\sigma}{N} \sum_{j} d ( R W^{j})_{t},  \quad R m^N_{0}= R \bar \xi^{N}, \\
d R \eta_{t}^{N}= \sum_{j}   \zeta_{t}^{j,N}  d ( R W^{j})_{t}, \quad  R \eta_{T}^{N}=  \Big( \I_d + \frac{1}{N} \nabla^{2} g( R m^N_T) \Big) \Big( R m^N_T +  \nabla g ( R m^N_T) \Big),
\end{array}\right.
$$
which is the same system as above, but with driving noises $(R \bar \xi ^N, R 
W ^1, ..., R 
W ^N )$. 
Since $\mathbb P \circ (R \bar \xi ^N, R W ^1, ..., R 
W ^N)^{-1} = \mathbb P \circ ( \bar \xi ^N,  W^1,..., W ^N)^{-1}$ by assumption, the uniqueness of the solution to the FBSDE (which follows exactly as in Theorem \ref{lemma existence uniqueness FBSDE.mean})  implies that 
$$
\mathbb  P \circ (R m^N)^{-1} = \mathbb  P \circ (- m^N)^{-1}.
$$
Therefore, also the distribution of $m^N_t$ is rotation invariant, and
for every cone $V$ (i.e., a Borel set $V \subset \R^d$ such that $s v \in V$ for any $s \in [0,\infty)$ and $v \in V$)  we have 
$$
\mathbb P [m^N_t \in V ] = \lambda_{d-1} (V \cap S^{d-1}),
$$
where $\lambda_{d-1}$ denotes the Lebesgue spherical measure on $S^{d-1}$.

Let $\mu \in \mathcal P (C([0,T], \R^d))$ be the limit of a subsequence (not relabeled) of $\mathbb P \circ (m^N)^{-1}$. 
Assume $\mu = \mathbb P \circ (m)^{-1}$, for a continuous stochastic process $m$. 
By Portmanteau Theorem, for every open cone $V$ we deduce that
$$
\begin{aligned}
    \mathbb P [m_t \in V ] &\leq \liminf_N \mathbb P [m^N_t \in V] =  \lambda_{d-1} (V \cap S^{d-1}). 
\end{aligned}
$$
Thus, for every cone $\tilde V$, the outer regularity of the spherical measure implies that
$$
\begin{aligned}
    \mathbb P [m_t \in \tilde V ] &\leq \inf \{  \mathbb P [m_t \in V ] \, | \, \tilde V \subset V, V \text{open} \} \\ 
    & \leq \inf \{ \liminf_N \mathbb P [m^N_t \in V] \, | \, \tilde V \subset V, V \text{open} \} \\ 
    &
    =  \inf \{  \lambda_{d-1} (V \cap S^{d-1}) \, | \, \tilde V \subset V, V \text{open} \}  = \lambda_{d-1} (\tilde V \cap S^{d-1}). 
\end{aligned}
$$
Thus, $m_t$ is distributed as the spherical measure; that is, 
$$
\begin{aligned}
    \mathbb P [m_t \in V ] =  \lambda_{d-1} (V \cap S^{d-1}). 
\end{aligned}
$$

By looking at the concavity/convexity of the function $a_1 \mapsto U (0,0; (a_1,0,...,0))$, we see that such a map has a unique minimum point $\hat a$ on $[0,\infty)$.
Thus, the minimizers of the control problem \eqref{eq:OC.simple} are constant controls $\theta \hat a$, with $\theta \in S^{d-1}$.
Hence, by Theorem \ref{thm:convergence}, $m$ is supported in the set $\{ m^\theta \, | \, m^\theta_t = \theta \hat a \, t, \ \theta \in S^{d-1} \}$.
Hence
$$
\begin{aligned}
    \mathbb P [m_t \in V ] = \mathbb P \Big[\frac{m_t}{\hat a t} \in V \cap S^{d-1} \Big] = \lambda_{d-1} (V \cap S^{d-1}). 
\end{aligned}
$$
Finally, since every subsequence have the same limit,  the sequence $(m^N)_N$ converges to $m$, thus completing the proof.

The proof that $(\alpha^0,m^0)=(0,0)$ is a mean field equilibrium which is not selected as limit of Nash equilibria is identical as in the previous proposition.
\end{proof} 

\begin{appendix}

\section{} 
    \subsection{Proof of Theorem \ref{lemma existence uniqueness FBSDE.mean}}\label{appendix lemma existence uniqueness FBSDE.mean}
   
    For simplicity of notation, we write $(m,\eta,\zeta)$ for $(m^N,\eta^N,\zeta^N)$.
    Define the function $\phi : [0,T] \to \R ^{d \times d}$ as the unique (symmetric) solution to the matrix valued system of Riccati ODE
    $$
    \begin{cases}
        \dot{\phi_t} &= \phi_t ^2  - \phi_t b - b^\tra \phi_t - \I_d, \\
        \phi_T &= \I_d,
    \end{cases}
    $$
    as ensured by \cite{yong1999linear}.
    Define the process $(\tilde \eta, \tilde \zeta)$ as
    $$
    \tilde \eta _t := \eta_t - \phi_t m_t
    \quad \text{and} \quad
    \tilde \zeta ^j_t := \zeta^j_t - \frac{\sigma}N \phi_t, \ j=1,...,N.
    $$
    If $(m,\eta,\zeta)$ solves the FBSDE \eqref{eq:FBSDE:mean}, using Ito formula we see that $(\tilde \eta, \tilde \zeta)$ solves the backward SDE 
$$
\begin{cases}
d \tilde \eta_{t} &= - \Big[ (b^\tra - \phi_t) \tilde \eta_t +  \frac{1}{N} \nabla^{2} f(m_t)  m_t  + \Big( \I_d + \frac{1}{N} \nabla^{2} f(m_t) \Big)   \nabla f (m_t)  \Big] dt +\sum_{j} \tilde \zeta_{t}^{j} d W_{t}^{j}, \\
\tilde \eta_{T}&=  \frac{1}{N} \nabla^{2} g(m_T)  m_T  + \Big( \I_d + \frac{1}{N} \nabla^{2} g(m_T) \Big)   \nabla g (m_T) .
\end{cases}
$$ 
Moreover, setting
$$
\tilde m _t := m_t \Phi_t, \quad \Phi_t :=  \exp \bigg( - \int_0^t (b-\phi_s)ds \bigg),
$$
we have
$$
d \tilde m _t = - \Phi_t \tilde \eta_t dt + \Phi_t \frac{\sigma}{N} \sum_j dW^j_t, \quad \tilde m_0 = \frac{1}{N}\sum_j \xi^j.
$$

Actually, one can easily verify that $(m,\eta,\zeta)$ solves the FBSDE \eqref{eq:FBSDE:mean} if and only if the process 
$(\tilde m,\tilde \eta,\tilde \zeta)$ solves the system of FBSDE
$$ 
\begin{cases}
d \tilde m _t &= - \Phi_t \tilde \eta_t dt + \Phi_t \frac{\sigma}{N} \sum_j dW^j_t,  \\
d \tilde \eta_{t} &= - \Big[ (b^\tra - \phi_t) \tilde \eta_t +  \frac{1}{N} \nabla^{2} f(\Phi_t^{-1} \tilde m _t )  \Phi_t^{-1} \tilde m _t   + \Big( \I_d + \frac{1}{N} \nabla^{2} f(\Phi_t^{-1} \tilde m _t ) \Big)   \nabla f (\Phi_t^{-1} \tilde m _t)  \Big] dt \\
&\quad +\sum_{j} \tilde \zeta_{t}^{j} d W_{t}^{j}, \\
\tilde m_0 & = \frac{1}{N}\sum_j \xi^j,
\qquad 
\tilde \eta_{T}=  \frac{1}{N} \nabla^{2} g( \Phi_T^{-1} \tilde m _T)  \Phi_T^{-1} \tilde m _T  + \Big( \I_d + \frac{1}{N} \nabla^{2} g(\Phi_T^{-1} \tilde m _T) \Big)   \nabla g (\Phi_T^{-1} \tilde m _T) ,
\end{cases}
$$ 
where $\Phi_t^{-1}$ denotes the inverse of the matrix $\Phi_t$.

Thanks to our assumptions on the boundedness and on the Lipschitzianity of the functions $\nabla f (y), \nabla^2 f (y), \nabla^2 f(y) y,\nabla g (y), \nabla^2 g (y), \nabla^2 g(y) y$, and thanks to the non-degeneracy (which follows from the boundedness of $\phi$), Theorem 2.6 in \cite{delarue.02} provides the existence of a unique solution  $(\tilde m,\tilde \eta,\tilde \zeta)$ to the previous FBSDE. 
Hence, the inverse transformation gives the existence of a unique solution $( m, \eta, \zeta)$ to \eqref{eq:FBSDE:mean}, which completes the proof.

\subsection{Proof of Theorem \ref{thm:EU NE}}\label{appendix thm:EU NE}
The argument consists in constructing the solution  to \eqref{eq:FBSDE:N} starting from the solution of \eqref{eq:FBSDE:mean}, which by Theorem \ref{lemma existence uniqueness FBSDE.mean} exists unique.

\emph{Step 1.} 
Let $(m^*, \eta^*, \zeta^*)$ be the unique solution to \eqref{eq:FBSDE:mean}, and consider the FBSDE
\begin{equation}
\label{eq:FBSDE:N fixed m}
\left\{
\begin{array}{l}
d X_{t}^{i}=( b X^i_t-Y_{t}^{i}) d t+\sigma d W_{t}^{i}, \qquad X_{0}^{i}=\xi^{i},  \\
 d Y_{t}^{i}=- \big[ b^\tra Y^i_t + 
 \big( \I_d + \frac{1}{N} \nabla^{2} f(m^*_t) \big) ( X^i_t +  \nabla f (m^*_t)) \big] dt+\sum_{j=1}^{N} Z_{t}^{i j} d W_{t}^{j}, \\
Y_{T}^{i}=\big( \I_d + \frac{1}{N} \nabla^{2} g (m^*_T) \big) ( X^i_T +  \nabla g (m^*_T)).
\end{array}\right.
\end{equation}
For any fixed $i=1,\dots, N$, this is a linear FBSDE with random coefficients for the unknowns $(X^i, Y^i, (Z^{ij})_j)$ which can be seen to represent the optimality conditions of a linear-quadratic stochastic control problem with random coefficients.  The existence of a unique adapted solution is ensured by \cite[Theorem 2.2]{tang2003general} as soon as the matrices $\I_d + \frac1N \nabla^2 f(m^*_t)$ and $\I_d + \frac1N \nabla^2 g(m^*_T)$ are nonnegative semi-definite, which is the case for $N$ large enough; notably, the solution can be represented by a backward stochastic Riccati equation.

\emph{Step 2.} 
We next look at the empirical means related to $(\bm{X}, \bm{Y}, \bm{Z})$, namely 
$$
m^N_t = \frac{1}{N} \sum_{i=1}^N X_{t}^{i}, \quad  \eta_{t}^{N}=\frac{1}{N} \sum_{i=1}^N Y_{t}^{i}
\quad \text{and}\quad
\zeta_{t}^{j, N}=\frac{1}{N} \sum_{i=1}^N Z_{t}^{i j}, \ j=1,...,N,
$$
and want to show that $m^N = m^*$.
Indeed, averaging over the components of \eqref{eq:FBSDE:N fixed m}, we find that $(m^N,\eta^N, \zeta^N)$ solves the system
$$
\left\{\begin{array}{l} 
d m_{t}^{N}= (b m^N_t-\eta_{t}^{N} ) dt+\frac{\sigma}{N} \sum_{j} d W_{t}^{j},   \\
d \eta_{t}^{N}= - \big[ b^\tra \eta^N_t + \big( \I_d + \frac{1}{N} \nabla^{2} f(m^*_t) \big) ( m^N_t +  \nabla f (m^*_t) ) \big] dt +\sum_{j} \zeta_{t}^{j,N} d W_{t}^{j}, \\
m^N_{0}=\frac{1}{N} \sum_{j} \xi^{i}, \qquad \eta_{T}^{N}=\big( \I_d + \frac{1}{N} \nabla^{2} g(m^*_T) \big) ( m^N_T +  \nabla g (m^*_T) ).
\end{array}\right.
$$ 
Such a system is equivalent to the FBSDE \eqref{eq:FBSDE:N fixed m}, hence it admits a unique strong solution. 
Noticing that $(m^*, \eta^*, \zeta^*)$ is also a solution to the previous system, we deduce that
$$(m^*, \eta^*, \zeta^*) = (m^N,\eta^N, \zeta^N),$$
which in turn implies that $(\bm{X}, \bm{Y}, \bm{Z})$ is the unique solution to the FBSDE \eqref{eq:FBSDE:N}.

\emph{Step 3.} Via a classical fixed point argument, one can construct a Nash equilibrium $\bm{\alpha} = (\alpha^1,...,\alpha^N)$ for a weak formulation of the game; the details are long but standard, and thus are omitted. Then,  stochastic maximum principle implies any equilibrium gives a solution to a weak formulation of the FBSDE \eqref{eq:FBSDE:N}. 
In light of the previous step, for $\sigma >0$ and $N$ large enough, system \eqref{eq:FBSDE:N} admits a unique strong solution and thus the equilibrium is  is unique, strong, and given by  $\alpha^i = -Y^i$ for any $i=1,...,N$.

\end{appendix}

\bibliographystyle{siam}
\bibliography{jodi.bib}

\end{document}